\documentclass[12pt,a4paper,reqno]{amsproc}
\usepackage{amsmath}
\usepackage{amssymb}
\usepackage{amsthm}
\usepackage[backrefs]{amsrefs}
\usepackage[german,danish,english]{babel}
\usepackage[utf8]{inputenc} 
\usepackage[T1]{fontenc}
\usepackage{graphicx}
\usepackage{color}
\usepackage{mathrsfs}
\usepackage{comment}
\usepackage{multirow}
\usepackage{enumitem}
\usepackage[footnote,draft,danish,silent]{fixme}
\usepackage{MnSymbol}

\newcommand{\NN}{\mathbb N}
\newcommand{\ZZ}{\mathbb Z}
\newcommand{\RR}{\mathbb R}

\newcommand{\EE}{\mathbb E}
\newcommand{\CC}{\mathbb C}
\newcommand{\del}{\subseteq}
\newcommand{\eps}{\varepsilon}
\newcommand{\ph}{\varphi}

\newcommand{\MM}{\mathcal M_{\infty}}
\newcommand{\Emax}{\mathbb E^{\mathrm{max}}}
\newcommand{\Emin}{\mathbb E^{\mathrm{min}}}

\newcommand{\tnorm}[1]{{\left\vert\kern-0.25ex\left\vert\kern-0.25ex\left\vert #1 
    \right\vert\kern-0.25ex\right\vert\kern-0.25ex\right\vert}}

\DeclareMathOperator{\Ra}{Re}

\DeclareMathOperator{\diag}{diag}

\DeclareMathOperator{\supp}{supp}
\DeclareMathOperator{\rank}{rank}
\DeclareMathOperator{\Tr}{Tr}

\newtheorem{stn}{Sætning}[section]
\newtheorem{theorem}[stn]{Theorem}
\newtheorem{cor}[stn]{Corollary}
\newtheorem{lemma}[stn]{Lemma}
\newtheorem{prop}[stn]{Proposition}
\newtheorem{question}[stn]{Question}
\theoremstyle{definition}\newtheorem{defi}[stn]{Definition}
\theoremstyle{remark}\newtheorem{rem}[stn]{Remark}

\title[Inverse and stability theorems for amenable groups]{Operator algebraic approach to inverse and stability theorems for amenable groups}
\author{Marcus De Chiffre}
\address{M. De Chiffre, TU Dresden, 01062 Dresden, Germany}
\email{marcus\_dorph.de\_chiffre@tu-dresden.de}

\author{Narutaka Ozawa}
\address{N. Ozawa, RIMS, Kyoto University, 606-8502 Kyoto, Japan}
\email{narutaka@kurims.kyoto-u.ac.jp}

\author{Andreas Thom}
\address{A. Thom, TU Dresden, 01062 Dresden, Germany}
\email{andreas.thom@tu-dresden.de}

\begin{document}
\begin{abstract}
We prove an inverse theorem for the Gowers $U^2$-norm for maps $G\to\mathcal M$ from an countable, discrete, amenable group $G$ into a von Neumann algebra $\mathcal M$ equipped with an ultraweakly lower semi-continuous, unitarily invariant (semi-)norm $\Vert\cdot\Vert$. We use this result to prove a stability result for unitary-valued $\eps$-re\-pre\-sen\-ta\-tions $G\to\mathcal U(\mathcal M)$ with respect to $\Vert\cdot \Vert$.
\end{abstract}

\maketitle

\tableofcontents

\section{Introduction}

\subsection{Inverse theorems}

The uniformity norms first appeared in Gowers' work on arithmetic progressions \cite{gow} and \cite{gow2}. For a map $\ph\colon G\to \CC$ defined on a finite abelian group, the second uniformity norm is given by
\[\Vert \ph\Vert_{U^2}=\frac{1}{|G|^3}\sum_{x-y+z-w=0}\ph(x)\overline{\ph(y)}\ph(z)\overline{\ph(w)}.\]
It is simple to show that a function with large $U^2$-norm is correlated to some character $\chi\colon G\to\mathcal U_1\del \CC$ in the sense that
$\frac{1}{|G|}\sum_{x\in G}\ph(x)\overline{\chi(x)}$
is relatively large. This result was generalized by Gowers and Hatami to non-abelian finite groups and matrix-valued maps. More precisely, they proved the following theorem.

\begin{theorem}[Gowers-Hatami, \cite{gh}]\label{ghmain}
Let $G$ be a finite group, let $c\in [0,1]$ and let $\ph\colon G\to\mathbb M_n$ with $\Vert\ph(x)\Vert_{\mathrm{op}}\leq 1$ such that
\[\Vert\ph\Vert_{U^2}:=\frac{1}{|G|^3}\sum_{xy^{-1}zw^{-1}=e}\Tr(\ph(x)\ph(y)^*\ph(z)\ph(w)^*)\geq cn.\]
Then there are $m\in [\frac{c}{2-c}n,\frac{2-c}{c}n]$, $\pi\colon G\to\mathcal U_m$ and maps $U,V\colon \CC^n\to\CC^m$ such that
\[\frac{1}{|G|}\sum_{x\in G}\Tr(\ph(x)V^*\pi(x)^*U)\geq t(c)n,\]
where $t(c)=\max\big\{\frac{c^8}{(2-c)^8},\frac{c^2}{4}\big\}.$
Moreover, if $n\leq m$ we can take $U$ and $V$ to be isometries and if $n\geq m$ we can take $U$ and $V$ to be co-isometries.
\end{theorem}

The proof of Theorem \ref{ghmain} presented in \cite{gh} relies heavily on properties of the Fourier transform of $\ph$, which has to be generalised in two ways from the usual Fourier transform since $\ph$ is defined on a possibly non-abelian group and takes matrix values. This Fourier theory works well for finite (or compact) groups, but it is hard to generalize to infinite, discrete groups. 
Our interest in the above theorem is in particular due to the fact that it can be used to prove a certain stability result for $\eps$-representations with respect to the $p$-norm coming from the trace on $\mathbb M_n$, which we will explain in depth later.
In a private communication with Gowers and Hatami, the second named author of this paper provided an alternative, operator algebraic proof of this stability result. A strength of this approach is that it is possible to generalize the proof, replacing $G$ with an (infinite) amenable group and $\mathbb M_n$ with a semi-finite von Neumann algebras $\mathcal M$. Gowers and Hatami ask in \cite{gh} whether this approach can be accomodated to a proof of their above theroem. In this article, we answer this question affirmatively -- this is our main result and can be found in Theorems \ref{main}, \ref{invtrace1} and \ref{invtrace2} which constitute generalizations of Theorem \ref{ghmain} to various settings. 

Our main theorem (Theorem \ref{main}) is formulated in a quite general way and it might be hard to compare the proofs, so for clarity's sake, we will now outline the main differences between them. Gowers and Hatami use the singular value decomposition of the Fourier transform of $\ph$ to cherry-pick the irreducible representations that are correlated to $\ph$ and put these together to a representation $\pi$ together with maps $U,V\colon \CC^n\to\CC^m$ such that $\frac{1}{|G|}\sum_{x\in G}\Tr(\ph(x)V^*\pi(x)^*U)$ is comparatively big. A subtle difficulty in their proof is that they have no control over the operator norms of $U$ and $V$ so they make use of a series of clever arguments to alter $U$ and $V$ so that they satisfy $\Vert U\Vert_{\mathrm{op}},\Vert V\Vert_{\mathrm{op}}\leq 1$.  Once these estimates are achieved, an extreme point argument provides the necessary isometries or co-isometries.
On the other hand, the crux of our proof is to use the Stinespring dilation theorem, which is a fundamental theorem about completely positive maps. It immediately brings into existence $\pi$ and the maps $U,V:\CC^n\to\CC^m$ so that $\frac{1}{|G|}\sum_{x\in G}\Tr(\ph(x)V^*\pi(x)^*U)$ is big, but the main advantage of using Stinespring's theorem is that it automatically gives us the estimates $\Vert U\Vert_{\mathrm{op}},\Vert V\Vert_{\mathrm{op}}\leq 1$ which makes the proof considerably shorter.

\subsection{Stability theorems}
As mentioned, the inverse theorem can be used to prove a so-called stability theorem for $\eps$-re\-pre\-sen\-ta\-tions.
Let us define what we mean by an $\eps$-re\-pre\-sen\-ta\-tion, or, more generally, an $\eps$-ho\-mo\-morph\-ism.
\begin{defi}
Let $\eps>0$, let $G$ and $H$ be groups and let $d$ be a metric on $H$. An $\eps$-\emph{homomorphism} is a map $\ph\colon G\to H$ such that
\[d(\ph(gh),\ph(g)\ph(h))<\eps,\qquad \forall g,h\in G.\]
\end{defi}
In accordance with nomenclature for ordinary homomorphisms, we call $\ph$ an $\eps$\emph{-representation} if $H$ consists of operators on a Hilbert space.
A natural and very general question about $\eps$-homomorphisms, which for instance was stated by Ulam in \cite{ulam}, can be formulated in the following way.

\begin{question}\label{juppo}
Consider a class $\mathscr G$ of groups together with a class $\mathscr H$ of metric groups. Given $\delta>0$ is there a $\eps>0$ such that for all $G\in\mathscr G$ and $(H,d)\in\mathscr H$ and $\eps$-homomorphisms $\ph\colon G\to H$ there is a genuine homomorphism $\pi\colon G\to H$ such that $d(\ph(g),\pi(g))<\delta$ for all $g\in G$?
\end{question}
If the question has a positive answer, we say that the class $\mathscr G$ is \emph{stable} with respect to to $\mathscr H$.
This question has been studied in various settings, and the answer highly depends much on which classes of groups and metrics one considers. The case where $\mathscr G$ is the class of finite (or, more generally, compact) groups and $\mathscr H$ consists of unitary groups equipped with the metric induced from the operator norm was treated in \cite{gkr}. Later Kazhdan generalized this result to amenable groups.
\begin{theorem}[Kazhdan, \cite{art:kazh}]\label{kazh}
Let $0<\eps <\frac{1}{200}$, let $G$ be a countable, discrete, amenable group and let $\mathcal H$ be a Hilbert space. Let $\ph\colon G\to\mathcal U(\mathcal H)$ be an $\eps$-rep\-re\-sen\-ta\-tion with respect to the metric coming from the operator norm $\Vert\cdot\Vert_{\mathrm{op}}$. Then there exists a representation $\pi\colon G\to\mathcal U(\mathcal H)$ such that
\[\Vert \ph(g)-\pi(g)\Vert_{\mathrm{op}}<2\eps,\qquad g\in G.\]
\end{theorem}
On the other hand P. Rolli \cite{rolli} gave an easy construction of non-trivial 1-dimensional $\eps$-representations of the free group on two generators, that is, a family $\ph_\eps\colon\mathbb F_2\to\mathcal U_1$ of $\eps$-representations ($\eps>0$) uniformly bounded away from the set of genuine representations. This construction was used in \cite{bto} to prove the existence of non-trivial $\eps$-representations $G\to\mathcal U(\mathcal H)$ of any group $G$ containing a free group on some (in general infinite dimensional) Hilbert space $\mathcal H$ with respect to the operator norm. More generally, they proved the existence of non-trivial $\eps$-representations for groups $G$ such the map $H^2_b(G,\RR)\to H^2(G,\RR)$ from bounded cohomology to usual cohomology is not injective. To the best of our knowledge, it is still open whether stability of $\eps$-representations is a characterizing property for amenable groups. Although we will not discuss this question here, let us state it properly.

\begin{question}
Does there exist a non-amenable group $G$ such that for all $\delta>0$ there exists $\eps>0$ such that if $\ph\colon G\to \mathcal U(\mathcal H)$ is an $\eps$-representation with respect to the operator norm there is a genuine representation $\pi\colon G\to\mathcal U(\mathcal H)$ satisfying $\Vert \ph(g)-\pi(g)\Vert_{\mathrm{op}}<\delta$?
\end{question}
It is also worth mentioning that a version of Question \ref{juppo} was considered in \cite{johnson} for Banach algebras and $\eps$-multiplicative functionals.

In this paper, our focus is $\eps$-representations of amenable groups with respect to some norm different from the operator norm. An important and motivating example is the case where $(\mathcal M,\tau)$ is a von Neumann algebra equipped with a faithful, normal trace and $\ph\colon G\to\mathcal U(\mathcal M)$ is an $\eps$-representation with repsect to the $p$-norm $\Vert T\Vert_p=\tau(|X|^p)^{1/p}$.
Gowers and Hatami use their inverse theorem to prove a stability result for $\eps$-representations $\ph\colon G\to \mathcal U_n$ in the case where $G$ is finite and $\mathcal U_n$ is equipped with the $p$-norm coming from the trace. For $p=2$, their result can be stated in the following way.
\begin{theorem}[Gowers-Hatami, \cite{gh}]\label{gohat}
Let $G$ be a finite group, let $0<\eps<\frac{1}{16}$ and let $\ph\colon G\to\mathcal U_n$ be an $\eps$-representation with repsect to the normalized $2$-norm $\Vert\cdot\Vert_2$. Then there exists $m\in [n,(1-4\eps^2)^{-1}n]$ together with a representation $\pi\colon G\to\mathcal U_m$ and an isometry $U\colon \CC^n\to\CC^m$ such that
\[\Vert\ph(g)-U^*\pi(g)U\Vert_2<42\eps,\qquad g\in G.\]
\end{theorem}
In the same way as Gowers and Hatami, we deduce a stability result for amenable groups with respect to unitary groups of von Neumann algebras equipped with any unitarily invariant, ultraweakly lower semi-continuous semi-norm (see Definition \ref{normi}). In particular, our result encompasses the results mentioned in this section.
Note that the above theorem does not quite answer Question \ref{juppo}.
If $n<\frac{1-4\eps^2}{4\eps^2}$, it follows that $m=n$, but, as Gowers and Hatami point out, in order to get a result which holds uniformly for all $n$ and $\eps$ idependently, one needs to allow the dimension of the approximating representation $\pi$ to differ from the dimension of $\ph$. This is, loosely speaking, because the normalised trace norm is insensitive to low-dimensional pertubations.
More precisely, let $G$ be a countable, discrete, amenable group with left-invariant mean $\EE$ and let $\mathcal M$ be a finite factor equipped with the 2-norm $\Vert\cdot\Vert_2$ associated with the faithful, normal tracial state. Let $\pi\colon G\to\mathcal U(\mathcal M)$ be a representation such that $\pi(G)$ generates $\mathcal M$ as a von Neumann algebra and let $P\in\mathcal M$ be a projection with $\tau(P)=1-\eps$ for some $0<\eps<\frac{1}{4}$. Then the cutdown $\ph\colon G\to P\mathcal M P$ given by $\ph(g)=P\pi(g)P,g\in G$ satisfies $\Vert \ph(gh)-\ph(g)\ph(h)\Vert_{2}<\eps$, but cannot be close to any representation $\rho\colon G\to\mathcal U(P\mathcal MP)$. Indeed, if such a $\rho$ existed with $\Vert \ph(g)-\rho(g)\Vert_2<\frac{1}{2}$ for all $g\in G$, then the operator $\EE_x\rho(x)\pi(x)^*\in P\mathcal M$ would be a non-zero intertwiner of the representations $\rho$ and $\pi$, but since $\pi$ is a factor representation, this is only possible if $P=1_\mathcal M$ which is not the case since $\eps>0$. To paraphrase Gowers and Hatami: the representation that approximates $\ph$ is in some sense $\pi$, but $\pi$ happens to be of the wrong dimension.
Note that $\ph$ does not take values in $\mathcal U(P\mathcal MP)$, but this can be corrected for since $\Vert \ph(g)^*\ph(g)-P\Vert_2 <\eps$. Also, the fact that $\Vert\cdot \Vert_2$ is not normalized on $P\mathcal MP$ can be corrected for by replacing $\Vert\cdot\Vert_2$ with $\frac{1}{\sqrt{1-\eps}}\Vert \cdot\Vert_2$.
Bottom line is, that if $G$ is an amenable group with either irreducible representations of arbitrarily high dimension or a representation whose image generates a (necessarily hyperfinite) $\mathrm{II}_1$ factor, then there are non-trivial $\eps$-representations of $G$, but the above result of \cite{gh} and our generalization, Theorem \ref{epso}, show that the above construction is the only way of producing non-trivial $\eps$-representations of amenable groups.

\section{Preliminaries}\label{prol}
Throughout this article, all groups considered are assumed to be discrete and countable. Hilbert spaces are assumed separable and thus von Neumann algebras have separable preduals. For a Hilbert space $\mathcal H$, we let $\mathbb B(\mathcal H)$ denote the algebra of bounded linear operators on $\mathcal H$. If $\dim\mathcal H=n$, we write $\mathbb M_n=\mathbb B(\mathcal H)$.
Given a subset $\mathcal S\del\mathbb B(\mathcal H)$, we let $\mathcal S'$ denote the \emph{commutant} of $\mathcal S$, that is, the set of operators $T\in\mathbb B(\mathcal H)$ that commute with all $S\in \mathcal S$, i.\,e. $ST=TS$. For a projection $P\in\mathbb B(\mathcal H)$, we let $P^\perp:=1_{\mathcal H}-P$ be the projection on the orthogonal complement of $P\mathcal H$. 
Whenever we consider an amenable group $G$, we will implicitly fix a \emph{symmetric bi-invariant mean} $\mathbb E\in\ell^\infty(G)^*$, and we shall write $\mathbb E_xf(x):=\mathbb E(f)$, i.e., $\EE(1)=1$, $\EE(f)\geq 0$ for $f\geq 0$ and
\[\EE_xf(x)=\EE_xf(gx)=\EE_xf(xg)=\EE_xf(x^{-1}),\qquad g\in G\tag{\dag}\label{li}.\]

Let $\mathcal M$ be a von Neumann algebra. We let $\mathcal U(\mathcal M)$ denote the unitary group of $\mathcal M$. If $\mathcal M=\mathbb M_n$, we write $\mathcal U_n=\mathcal U(\mathcal M)$. Since we will be dealing different norms, we will use $\Vert\cdot\Vert_{\mathrm{op}}$ to denote the operator norm on $\mathcal M$. We define $\MM  :=\mathcal M\bar\otimes \mathbb B(\ell^2(\NN))$
and view $\mathcal M$ as a corner of $\MM$. More precisely, we implicitly fix a rank 1-projection $E\in\mathbb B(\ell^2(\NN))$ and identify 
\[\mathcal M\simeq (1_\mathcal M\otimes E)\mathcal M_\infty(1_\mathcal M\otimes E),\]
where $1_\mathcal M$ is the unit of $\mathcal M$. Consistent with this identification, we write $1_\mathcal M$ instead of $1_\mathcal M\otimes E$. We denote the unit of $\MM  $ by $1_\infty$.
Recall that  $\mathcal M$  is the dual of the Banach space of its normal functionals, which we will denote $\mathcal M_*$.
The associated weak* topology on $\mathcal M$ is called the \emph{ultraweak} or $\sigma$\emph{-weak} topology.
Given an amenable group $G$ and a map $\ph\colon G\to \mathcal M$ such that $\sup_{x\in G}\Vert \ph(x)\Vert_{\mathrm{op}}<\infty$, we can define the mean
$\mathbb E_x\ph(x)\in\mathcal M\simeq (\mathcal M_*)^*$ by the formula
\[f(\mathbb E_x\ph(x))=\mathbb E_x f(\ph(x)),\qquad f\in\mathcal M_*.\]
The characteristic function on an interval $[a,b]\del \RR$ we denote by $\chi_{[a,b]}$.
We also recall the following definition.
\begin{defi}
Let $G$ be a group and let $\mathcal H$ be a Hilbert space. A map $\ph\colon G\to\mathbb \mathbb B(\mathcal H)$ is called \emph{positive definite} if for all finite sets $F\del G$, the matrix $[\ph(xy^{-1})]_{x,y\in F}\in\mathbb B(\mathcal H^{\oplus|F|})$ is positive as an operator, i.e.,
if
\[\sum_{x,y\in F}\langle \ph(xy^{-1})\xi_y,\xi_x\rangle\geq 0,\]
for all $\xi_x\in \mathcal H,x\in F$.
\end{defi}

We start out with a rather simple observation, which turns out to be central to this paper. The relevance of the following proposition to stability of $\eps$-representations was noted by Shtern in \cite{shtern}.
\begin{prop}\label{snakeispos}
Let $G$ be an amenable group and let $\ph\colon G\to  \mathcal M$ be given such that $\sup_{x\in G}\Vert\ph(x)\Vert_\infty<\infty$. Then the map $\tilde\ph\colon G\to\mathcal M$ defined by
\[\tilde\ph(x)=\mathbb E_y\ph(xy)\ph(y)^*,\qquad x\in G\]
is positive definite.
\end{prop}
\begin{proof}
Let $F\del G$ be finite and let $\xi_x\in \mathcal H,x\in F$. Then 
\begin{align*}
\sum_{x,y\in F}\langle \tilde\ph(xy^{-1})\xi_y,\xi_x\rangle&=\sum_{x,y\in F}\mathbb E_z\langle \ph(xy^{-1}z)\ph(z)^*\xi_y,\xi_x\rangle
\\&=\sum_{x,y\in F}\mathbb E_z\langle\ph(xz)\ph(yz)^*\xi_y,\xi_x\rangle
\\&=\sum_{x,y\in F}\mathbb E_z\langle\ph(yz)^*\xi_y,\ph(xz)^*\xi_x\rangle
\\&=\mathbb E_z\langle\sum_{y\in F}\ph(yz)^*\xi_y,\sum_{x\in F}\ph(xz)^*\xi_x\rangle\geq 0,
\end{align*}
since $\langle \xi,\xi\rangle\geq 0$ for all $\xi\in \mathcal H$ and $\mathbb E$ is positive.
\end{proof}
A fundamental fact about positive definite maps is \emph{Stinespring's dilation theorem}. We will use a formulation which is essentially Theorem 3 in \cite{kasp}. For the reader's convenience, we recall the proof.
\begin{theorem}[Kasparov, \cite{kasp}]\label{sdt}
Let $G$ be a group and let $\mathcal M$ be a von Neumann algebra.
For every positive definite map $$\ph\colon G\to\mathcal M\simeq 1_{\mathcal M}\MM 1_\mathcal M$$ there exist $U\in \MM  1_\mathcal M$ and a representation $\pi\colon G\to\mathcal U(\MM  )$ such that
\[\ph(g)=U^*\pi(g)U,\qquad g\in G.\]
In particular $\Vert U\Vert_{\mathrm{op}}^2=\Vert\ph(1)\Vert_{\mathrm{op}}.$
\end{theorem}
\begin{proof}
Let $\mathcal M\del\mathbb B(\mathcal H)$ be a normal representation of $\mathcal M$ and consider the
 vector space $\mathcal A:=C_{\mathrm{fin}}(G,\mathcal H)$ of finitely supported maps $G\to\mathcal H$ equipped with the sequilinear form
\[\langle f,g\rangle_\ph=\sum_{x,y\in G} \langle \ph(y^{-1}x)f(x),g(y)\rangle_{\mathcal H},\]
for $f,g\in \mathcal A$. By positive definiteness of $\ph$, this is a positive semidefinite sequilinear form, so by separation and completion, we get a Hilbert space
$\mathcal{\tilde{H}}$ where $G$ acts as unitaries by the formula
\[\pi_0(g)[f]=[g.f], \qquad f\in \mathcal A,\]
where $g.f(x):=f(g^{-1}x)$ is the left translation action and $[f]$ denotes the equivalence class of $f$.
Furthermore, let $U_0\colon \mathcal H\to \mathcal{\tilde H}$ be given by
\[U_0(\xi)=[\delta_e\xi],\qquad \xi\in\mathcal H.\]
One sees straightforwardly that $U_0^*([f])=\sum_{x\in G}\ph(x)f(x)$ for $f\in\mathcal A$
and from this it is clear that
\[\ph(g)=U_0^*\pi_0(g)U_0.\]
We also define an action of the commutant $\mathcal M'\del\mathbb B(\mathcal H)$ on $\mathcal A$ by
\[\rho(T)[f]=[Tf],\qquad T\in\mathcal M',f\in\mathcal A.\]
In order to extend $\rho$ to a normal representation of $\mathcal M'$ on $\mathcal{\tilde H}$, we have to check that $\rho$ is well-defined and bounded. Note that since $T$ commutes with $\ph$ and $[\ph(xy^{-1})]_{x,y\in F}$ is positive for any $F$, the operator $S:=\diag(T^*)[\ph(xy^{-1})]_{x,y\in F}\diag(T)$ is positive; in fact $0\leq S\leq \Vert T\Vert^2_{\mathrm{op}}[\ph(xy^{-1})]_{x,y\in F}$, so
\begin{align*}
\Vert[Tf]\Vert_{\mathcal{\tilde H}}^2&=\sum_{x,y\in F}\langle T^*\ph(y^{-1}x)Tf(x),f(y)\rangle\\
&=\langle Sf,f\rangle_\ph
\leq \Vert T\Vert_{\mathrm{op}}^2\sum_{i,j}\langle\ph(y^{-1}x)f(x),f(y)\rangle,
\end{align*}
where $F=\supp f$,
so $\rho(T)$ extends to an operator on $\mathcal{\tilde H}$. It is now easy to see that $\rho$ is a normal representation of $\mathcal M'$. Thus, by the representation theory for von Neumann algebras, there is an isometry $V:\mathcal{\tilde H}\to \mathcal H\otimes \ell^2$ such that $\rho(T)=V^*(T\otimes 1)V$ and $VV^*\in(\mathcal M'\otimes 1)'$. Clearly $\pi_0(G)\del\rho(\mathcal M')'$, so
\[\pi(g):=V\pi_0(g)V^*+ 1-VV^*\in (\mathcal M'\otimes 1)'=\MM  ,\]
is a unitary representation of $G$ which, together with the map $U=VU_0$, has the desired properties.
\end{proof}

We will consider a special class of semi-norms on the von Neumann algebra $\mathcal M$.

\begin{defi}\label{normi}
Let $\mathcal M$ be a von Neumann algebra. A \textbf{unitarily invariant semi-norm} on $\mathcal M$ is a semi-norm $\Vert\cdot\Vert$ on an (algebraic) ideal $\mathcal A\del\mathcal M$ such that for all $U,V\in\mathcal U(\mathcal M)$ and $T\in\mathcal M$ it holds that
\[\Vert UTV\Vert=\Vert T\Vert.\tag{$\circ$}\label{ui}\]

The semi-norm $\Vert\cdot\Vert$ is called \textbf{ultraweakly lower semi-continuous} if the unit ball $\{T\in \mathcal M\mid \Vert T\Vert\leq 1\}$ is closed in the ultraweak toplogy.
\end{defi}

We consider such seminorms as defined on all of $\mathcal M$ by assigning the value $\infty$ outside of the ideal $\mathcal A$.
An important example of unitarily invariant ultraweakly lower semi-continuous semi-norms occurs in the case where $\mathcal M$ is a semi-finite von Neumann algebra equipped with a normal trace $\tau$. In this case, we define the $p$-semi-norms by
\[\Vert T\Vert_p:=\tau(|T|^{p})^{1/p},\qquad T\in\mathcal M.\]
The tracial property implies unitary invarance and the fact that $\tau$ is normal implies that $\Vert\cdot\Vert_p$ is ultraweakly lower semi-continuous. If $\tau$ is faithful, this is a norm.
In the following, we will list the basic properties of ultraweakly semi-continuous, unitarily invariant semi-norms that we will use throughout this paper. We start by a basic proposition.

\begin{prop}\label{headache}
Let $R,S\in\mathcal M$ and assume $0\leq R\leq S$. Then there exists $T\in\mathcal M$ with $\Vert T\Vert_{\mathrm{op}}\leq 1$ and $R=S^{1/2}TS^{1/2}$.
\end{prop}
\begin{proof}
Let $A_n=\chi_{[\frac{1}{n},\infty)}(S^{1/2})S^{-1/2}$. We note that
\begin{align*}
\Vert R^{1/2}A_n\Vert_{\mathrm{op}}^2&
=\Vert A_nRA_n\Vert_{\mathrm{op}}
\leq \Vert A_nSA_n\Vert_{\mathrm{op}}
\\&=\Vert\chi_{[\frac{1}{n},\infty)}(S^{1/2})\Vert_{\mathrm{op}}\leq 1,
\end{align*}
so $R^{1/2}A_n$ has an ultraweak limit point, say $A\in\mathcal M$ with $\Vert A\Vert_{\mathrm{op}}\leq 1$. We also note that the increasing sequence $\chi_{[\frac{1}{n},\infty)}(S^{1/2})$ converges even strongly to $\chi_{(0,\infty)}(S^{1/2})$.
By a similar calculation as above, we have that
\[\Vert R^{1/2}\chi_{\{0\}}(S^{1/2})\Vert_{\mathrm{op}}\leq \Vert S^{1/2}\chi_{\{0\}}(S^{1/2})\Vert_{\mathrm{op}}=0,\]
so for some ultraweakly convergent subnet, we have that
\[AS^{1/2}=\lim_{\alpha}R^{1/2}\chi_{[\frac{1}{n_\alpha},\infty)}(S^{1/2})=R^{1/2}\chi_{(0,\infty)}(S^{1/2})=R^{1/2}.\]
Letting $T=A^*A$, we reach the desired conclusion.
\end{proof}

\begin{prop}\label{yum}
Let $\mathcal M$ be a von Neumann algebra and let $\Vert\cdot\Vert$ be a unitarily invariant semi-norm on $\mathcal M$. Then, for all $R,S,T\in\mathcal M$, we have that

\begin{gather}
\Vert RTS\Vert\leq \Vert R\Vert_{\mathrm{op}}\Vert T\Vert\Vert S\Vert_{\mathrm{op}},\tag{$\spadesuit$}\label{p1}
\\
\Vert T\Vert=\Vert T^*\Vert=\Vert |T|\Vert,\tag{$\diamondsuit$}\label{p2}
\\
\Vert T^*T\Vert=\Vert TT^*\Vert,\tag{$\clubsuit$}\label{p3}
\end{gather}
if $0\leq R\leq S,$ then
\begin{align}
\Vert R\Vert\leq \Vert S\Vert.\tag{$\heartsuit$}\label{p4}
\end{align}
\end{prop}

\begin{proof}
We begin with the proof of \eqref{p1}.
First assume $\Vert R\Vert_{\mathrm{op}},\Vert S\Vert_{\mathrm{op}}< 1$. By (a strengthening of) the Russo-Dye Theorem (see \cite{rkgkp}), $R$ and $S$ are convex combinations of unitaries in $\mathcal M$, that is, $R=\sum_{i=1}^n\lambda_iU_i$ and $S=\sum_{i=1}^m\mu_iV_i$ with $\lambda_i,\mu_i\in [0,1],U_i,V_i\in\mathcal U(\mathcal M)$ and $\sum_{i=1}^n\lambda_i=\sum_{i=1}^m\mu_i=1$.
Thus 
\[\Vert RTS\Vert\leq \sum_{i=1}^n\sum_{j=1}^m\lambda_i\mu_j\Vert U_iTV_j\Vert=\sum_{i=1}^n\lambda_i\sum_{j=1}^m\mu_j\Vert T\Vert=\Vert T\Vert.\]
Now let $R$ and $S$ be arbitrary and let $\eps>0$. Then $R':=(\Vert R\Vert_{\mathrm{op}}+\eps)^{-1}R$ and $S':=(\Vert S\Vert_{\mathrm{op}}+\eps)^{-1}S$ have operator norm strictly less than 1, so we get that
\[\Vert RTS\Vert=(\Vert R\Vert_{\mathrm{op}}+\eps)(\Vert S\Vert_{\mathrm{op}}+\eps)\Vert R'TS'\Vert\leq (\Vert R\Vert_{\mathrm{op}}+\eps)(\Vert S\Vert_{\mathrm{op}}+\eps)\Vert T\Vert.\]
Since this holds for all $\eps>0$, the result follows.

Now, for \eqref{p2}, by the polar decomposition, we have that
$T=U|T|$ and $|T|=U^*T$ for a parital isometry $U\in\mathcal M$. Thus, according to $(\spadesuit)$, we have that
\[\Vert T\Vert=\Vert U|T|\Vert\leq\Vert |T|\Vert=\Vert U^*T\Vert\leq \Vert T\Vert,\]
so $\Vert T\Vert=\Vert |T|\Vert$.
By taking adjoints on both sides of the equations, we also get that $\Vert T^*\Vert=\Vert |T|\Vert$.

Proceeding with \eqref{p3}, using the polar decomposition as above, we get
\[\Vert T^*T\Vert= \Vert |T|^2\Vert=\Vert |T||T|^*\Vert=\Vert U^*TT^*U\Vert\leq \Vert TT^*\Vert=\Vert U|T||T|U^*\Vert\leq\Vert T^*T\Vert.\]

Finally, we prove \eqref{p4}. Let $R\leq S$. By Lemma \ref{headache} we determine $T\in\mathcal M,\Vert T\Vert_{\mathrm{op}}\leq 1$ such that $R=S^{1/2}TS^{1/2}$.
Thus it follows from $\eqref{p3}$ and $\eqref{p1}$ that
\[\Vert R\Vert=\Vert S^{1/2}TS^{1/2}\Vert=\Vert T^{1/2}ST^{1/2}\Vert\leq \Vert S\Vert.\qedhere\]
\end{proof}

Some consequences, which we will use throughout the proofs, are the following.
\begin{cor}\label{smalcor}
Let $\mathcal M$ be a von Neumann algebra and let $\Vert\cdot\Vert \colon \mathcal M \to \mathbb R \cup \{\infty\}$ be a unitarily invariant semi-norm on $\mathcal M$.
Let $S,T,P\in \mathcal M$ with $\Vert S\Vert_{\mathrm{op}},\Vert T\Vert_{\mathrm{op}}\leq 1$ and $P\geq S^*S,T^*T$. Then
\begin{align}
\Vert P-S^*S\Vert,\Vert P-T^*T\Vert\leq 2\Vert P-S^*T\Vert,\tag{$\sharp$}\label{q1}
\end{align}
\end{cor}
\begin{proof}
Since $P-S^*S,P-T^*T$ and $(S-T)^*(S-T)$ are positive, we get that
\begin{align*}
0\leq P-S^*S&\leq P-S^*S+P-T^*T+(S-T)^*(S-T)
\\&=P-S^*S+P-T^*T+S^*S+T^*T-S^*T-T^*S
\\&=P-S^*T+P-T^*S.
\end{align*}
Similarly, we have
\[0\leq P-T^*T\leq P-S^*T+P-T^*S,\]
so by \eqref{p4} and \eqref{p2}, using that $(P-T^*S)^*=P-S^*T$ (since $P$ is positive), the result follows.
\end{proof}

\begin{cor}\label{john}
Let $\mathcal M$ be a von Neumann algebra, let $\Vert\cdot\Vert$ be a unitarily invariant semi-norm on $\mathcal M$ and let $S,T\in \mathcal M$. Then
\begin{align}
\Vert S^*T\Vert\leq \frac{1}{2}(\Vert S^*S\Vert+\Vert T^*T\Vert)= \frac{1}{2}(\Vert SS^*\Vert+\Vert TT^*\Vert).\tag{$\flat$}\label{q2}
\end{align}
\end{cor}
\begin{proof}
Using polar decomposition again, we can find an operator $U$ with $\Vert U\Vert_{\mathrm{op}}\leq 1$ so that $U^*S^*T\geq 0$ and $\Vert S^*T\Vert=\Vert U^*S^*T\Vert$. Note that this implies that $U^*S^*T=T^*SU$. Thus
\begin{align*}
0&\leq (SU-T)^*(SU-T)=U^*S^*SU+T^*T-U^*S^*T-T^*SU
\\&=U^*S^*SU+T^*T-2U^*S^*T.
\end{align*}
Thus $U^*S^*T\leq \frac{1}{2}(U^*S^*SU+T^*T)$, so the inequality follows from \eqref{p4} and \eqref{p1}. The last equality is \eqref{p3}.
\end{proof}

\begin{lemma}\label{partiso}
Let $\mathcal M$ be a von Neumann algebra, let $P,Q\in\mathcal M$ be projections and let $S\in P\mathcal MQ$ with $\Vert S\Vert_{\mathrm{op}}\leq 1$. Then there are partial isometries $V_1,V_2\in P\mathcal M Q$ such that
\[S=\frac{1}{2}(V_1+V_2).\]
\end{lemma}
\begin{proof}
We write $S=U|S|$ where $U\in P\mathcal MQ$ is a partial isometry. Now let $V_{\pm}:=|S|\pm i\sqrt{Q-|S|^2}$. We note that $V_\pm$ are unitaries in $Q\mathcal MQ$, and
\[S=\frac{1}{2}(UV_++UV_-).\]
Thus $V_1=UV_+$ and $V_2=UV_-$ are partial isometries with the desired properties.
\end{proof}

The only place where we use ultraweak lower semi-continuity is in the following lemma. In fact, this extra assumption on the semi-norm $\Vert\cdot\Vert$ is only necessary if the group $G$ is infinite.

\begin{lemma}
Let $G$ be an amenable group, let $\mathcal M$ be a von Neumann algebra, let $\Vert\cdot\Vert$ be a ultraweakly lower semi-continuous semi-norm on $\mathcal M$ and let $\ph\colon G\to\mathcal M$ such that $\sup_{x\in G}\Vert \ph(x)\Vert_{\mathrm{op}}<\infty$. Then 
\[
\Vert\mathbb E_x\ph(x)\Vert\leq\mathbb E_x\Vert \ph(x)\Vert.\tag{\ddag}\label{r1}
\]
\end{lemma}
\begin{proof}
For $\mu\in \ell^1(G)\del\ell^\infty(G)^*$ with $\mu(x)\geq 0,x\in G$, we can define $\mu(\ph):=\sum_{x\in G}\mu(x)\ph(x)$. We note that this sum converges in operator norm and hence also in the ultraweak topology. Furthermore, for finite $F\del G$, by the triangle inequality, we have that
\[\Vert \sum_{x\in F}\mu(x)\ph(x)\Vert\leq \sum_{x\in F}\mu(x)\Vert \ph(x)\Vert\leq \mu_x(\Vert \ph(x)\Vert).\]
By lower semi-continuity, we get that
$\Vert\mu(\ph)\Vert\leq \mu_x(\Vert \ph(x)\Vert).$

Now let $\mu_i\in\ell^1(G)$ be a net of positive functions with $\Vert \mu_i\Vert_1=1$ converging to $\mathbb E$ in the weak* topology on $\ell^\infty(G)^*$.
For all $f\in\mathcal M_*$ we have that
\[f(\mu_i(\ph))=f(\sum_{x\in G}\mu_i(x) \ph(x))=\sum_{x\in G}\mu_i(x)f(\ph(x))\to \mathbb E_xf(\ph(x)),\]
so
$\mu_i(\ph)$ converges to $\mathbb E(\ph)$ in the ultraweak topology, whence we conclude that
\[\Vert \mathbb E_x\ph(x)\Vert\leq \liminf_i \Vert\mu_i(\ph)\Vert\leq \liminf_i(\mu_i)_x(\Vert \ph(x)\Vert)=\mathbb E_x\Vert \ph(x)\Vert.\qedhere\]\end{proof}

This concludes the preliminary section and we turn our attention to the main theorem of this article.

\section{The main theorem}
We now state and prove our main theorem.
\begin{theorem}\label{main}
Let $\eps>0$, let $G$ be an amenable group, $\mathcal M$  a von Neumann algebra and let $\Vert\cdot\Vert$ be a unitarily invariant, ultraweakly lower semi-continuous semi-norm $\Vert\cdot \Vert$  on $\MM  $.
Let $\ph\colon G\to \mathcal M$ be any map and assume that $\Vert\ph(x)\Vert_{\mathrm{op}}\leq 1$ for all $x\in G$ and that 
\begin{gather*}\mathbb E_x\mathbb E_y\mathbb E_z \Vert 1_\mathcal M-\ph(x)\ph(y)^*\ph(yz)\ph(xz)^*\Vert<\eps,\\
\mathbb E_x\mathbb E_y\mathbb E_z \Vert 1_\mathcal M-\ph(xy)\ph(y)^*\ph(z)\ph(xz)^*\Vert<\eps.\end{gather*}

Then there exists a projection $P\in\MM  $, partial isometries $U,V\in P(\MM  )1_\mathcal M$ and a representation $\rho\colon G\to\mathcal U(P\MM P)$ such that
\[\mathbb E_x\Vert 1_\mathcal M-\ph(x)V^*\rho(x)^*U\Vert<44\eps,\]
and
\[
\Vert 1_\mathcal M-U^*U\Vert<20\eps,\qquad
\Vert P-UU^*\Vert< 15\eps,\qquad
\Vert P-VV^*\Vert<85\eps.
\]
\end{theorem}

Before commencing the proof, a few comments are in order. The two inequalities in the assumptions replace the assumption $\Vert \ph\Vert_{U^2}\geq cn$ in Theorem \ref{ghmain} and are more or less a direct adaptation of this latter inquality to our situation.
Note that for amenable groups, unlike for finite groups, the equality $\mathbb E_x\mathbb E_yf(x,y)=\mathbb E_y\mathbb E_xf(x,y)$ does not hold in general. 
For instance, in the case $G=\ZZ$ the function 
\[f(x,y)=\left\{\begin{matrix} 1,&\text{if } |x|\leq |y|,\\
0,&\text{otherwise}.\end{matrix}\right.\]
satisfies $\EE_x\EE_yf(x,y)=1$ and $\EE_y\EE_xf(x,y)=0$ for all invariant means $\EE$ on $\ell^\infty(\ZZ)$!
Therefore, the inequalities in the assumptions in the theorem are in general different. One reason for working with an abstract semi-norm instead of, say, the trace $p$-norm, is that the proof becomes conceptually simple; the long computations on the next couple of pages are nothing but repeated applications of the basic facts about ultraweakly lower semi-continuous, unitarily invariant semi-norms that we collected and proved in Section \ref{prol}. In order to underline this point, and hopefully to the convenience of the reader, we indicate the usage of the (in)equalities \eqref{li}, \eqref{ui}, \eqref{p1}, \eqref{p2}, \eqref{p3}, \eqref{p4}, \eqref{q1}, \eqref{q2} and \eqref{r1} to the right of the (in)equality, where it is used. Lest we forget the triangle inequality of $\Vert\cdot\Vert$, its usage will be indicated by ($\triangle$). The lines without any indications should be self-explanatory from the definitions or remarks during the proof.

\begin{proof}[Proof of Theorem \ref{main}]
Define $\tilde\ph\colon G\to\mathcal M$ by 
$$\tilde\ph(x):=\mathbb E_y\ph(xy)\ph(y)^*,x\in G,$$ which is positive definite by Proposition \ref{snakeispos}. Thus, by the Stinespring dilation theorem (Proposition \ref{sdt}), there exists a unitary representation $\pi\colon G\to\mathcal U(\MM  )$ together with $U\in \MM 1_\mathcal M$ with $\Vert U\Vert_{\mathrm{op}}\leq 1$ such that
\[\tilde\ph(g)=U^*\pi(g)U,\qquad g\in G.\]
Define $V:=\mathbb E_x \pi(x)^*U\ph(x)\in \MM 1_\mathcal M$ and $A:=\EE_x\pi(x)UU^*\pi(x)^*=\EE_x\pi(x)^*UU^*\pi(x)\in\pi(G)'$. We see that
\begin{align*}
\mathbb E_x\Vert 1_\mathcal M&-\ph(x)V^*\pi(x)^*U\Vert
\\&= \EE_x\Vert \EE_y( 1_\mathcal M-\ph(x)\ph(y)^*U^*\pi(yx^{-1})U)\Vert\\
&\leq \mathbb E_x\mathbb E_y\Vert 1_\mathcal M-\ph(x)\ph(y)^*U^*\pi(yx^{-1})U\Vert \tag{\ref{r1}}\\
&=\mathbb E_x\mathbb E_y\Vert 1_\mathcal M-\ph(x)\ph(y)^*\tilde\ph(yx^{-1})\Vert\\
&=\mathbb E_x\EE_y\Vert \EE_z (1_\mathcal M-\ph(x)\ph(y)^*\ph(yx^{-1}z)\ph(z)^*)\Vert
\\&\leq\mathbb E_x\mathbb E_y\mathbb E_z \Vert 1_\mathcal M-\ph(x)\ph(y)^*\ph(yx^{-1}z)\ph(z)^*\Vert\tag{\ref{r1}}
\\&=\mathbb E_x\mathbb E_y\mathbb E_z \Vert 1_\mathcal M-\ph(x)\ph(y)^*\ph(yz)\ph(xz)^*\Vert<\eps.\tag{\ref{li}}
\end{align*}
so $U$ and $V$ satisfy the desired inequality, but they need not be partial isometries and we have \textit{a priori} no control over their range projections.
In order to correct for that, we observe that $0\leq A\leq 1_{\infty}$ and
\begin{align*}
\Vert U^*(1_\infty-A)U\Vert
&\leq\mathbb E_x\Vert U^*U-U^*\pi(x)UU^*\pi(x)^*U\Vert\tag{\ref{r1}}
\\&= \EE_x\Vert U^*U-\tilde\ph(x)\tilde\ph(x)^*\Vert
\\&\leq \mathbb E_x\Vert 1_\mathcal M-\tilde\ph(x)\tilde\ph(x)^*\Vert\tag{\ref{p4}}
\\&\leq \mathbb E_x\mathbb E_y\mathbb E_z \Vert 1_\mathcal M-\ph(xy)\ph(y)^*\ph(z)\ph(xz)^*\Vert <\eps,\tag{\ref{r1}}
\end{align*}
so, since $\pi(x)$ and $(1-A)^{1/2}$ commute,
\begin{align*}
\Vert A-A^2\Vert&=\Vert (1_\infty-A)^{1/2}A(1_\infty-A)^{1/2}\Vert
\\&\leq \mathbb E_x\Vert(1_\infty-A)^{1/2}\pi(x)UU^*\pi(x)^*(1_\infty-A)^{1/2}\Vert\tag{\ref{r1}}
\\&=\mathbb E_x\Vert\pi(x)(1_\infty-A)^{1/2}UU^*(1_\infty-A)^{1/2}\pi(x)^*\Vert
\\&=\Vert(1_\infty-A)^{1/2}UU^*(1_\infty-A)^{1/2}\Vert\tag{\ref{ui}}
\\&=\Vert U^*(1_\infty-A)U\Vert<\eps.\tag{\ref{p3}}
\end{align*}
Now let $P:=\chi_{[1/2,1]}(A)$ which is a projection that commutes with $\pi$, so we can consider the representation $\rho:G\to \mathcal U(P\MM P)$ given by $\rho(g):=P\pi(g)P$.
Since $|\chi_{[1/2,1]}(t)-t|\leq 2(t-t^2)$ for all $t\in [0,1]$, we have that
\[\Vert P-A\Vert\leq 2\Vert A-A^2\Vert<2\eps,\tag{\ref{p4}}\]
and hence for all $x\in G$, we have
\begin{align*}
\Vert V^*&P^\perp\pi(x)^*P^{\perp}U\Vert 
\\&\leq \mathbb E_y\Vert \ph(y)^*U^*P^\perp\pi(yx^{-1})P^\perp U\Vert\tag{\ref{r1}}
\\&\leq \mathbb E_y\Vert U^*P^\perp\pi(yx^{-1})P^\perp U\Vert\tag{\ref{p1}}
\\&\leq\frac{1}{2}(\EE_y\Vert \pi(y)^*P^\perp UU^*P^\perp\pi(y)\Vert+\Vert\pi(x)^*P^\perp UU^*P^\perp\pi(x)\Vert)\tag{\ref{q2}}
\\&=\Vert P^\perp UU^*P^\perp\Vert\tag{\ref{ui}}
\\&=\Vert U^*P^\perp U\Vert\tag{\ref{p3}}
\\&\leq\Vert U^*(1_\infty-A)U\Vert+\Vert U^*(A-P)U\Vert<\eps+2\eps=3\eps.\tag{$\triangle$}
\end{align*}
Since $\pi(x)=\rho(x)+P^\perp\pi(x)P^\perp,x\in G$, we get that
\[\mathbb E_x\Vert 1_\mathcal M-\ph(x)V^*\rho(x)^*U\Vert < \mathbb E_x\Vert 1_\mathcal M-\ph(x)V^*\pi(x)^*U\Vert+3\eps<4\eps.\tag{$\triangle$,\ref{p1}}\]
Replace $U$ and $V$ by $U_0=PU,V_0=PV\in P\MM  1_\mathcal M$. We still need to turn $U_0$ and $V_0$ into partial isometries.
We write the polar decomposition of $U_0=S|U_0|$ and define $U_1:=S\chi_{[1/2,1]}(|U_0|)$. This is a partial isometry, and we calculate
\begin{align*}
\Vert U_1-U_0\Vert&\leq \Vert\chi_{[1/2,1]}(|U_0|)-|U_0|\Vert\tag{\ref{p1}}
\\&\leq 2\Vert|U_0|-|U_0|^2\Vert\tag{\ref{p4}}
\\&\leq 2\Vert 1_\mathcal M-U^*PU\Vert\tag{\ref{p4}}
\\&<2\Vert 1_\mathcal M-U^*AU\Vert+4\eps\tag{$\triangle$}
\\&\leq 2\mathbb E_x\Vert 1_\mathcal M-\tilde\ph(x)\tilde\ph(x)^*\Vert+4\eps<6\eps\tag{\ref{r1}}
,
\end{align*}
so 
\[\mathbb E_x\Vert 1_\mathcal M-\ph(x)V_0^*\rho(x)^*U_1\Vert< \mathbb E_x\Vert 1_\mathcal M-\ph(x)V_0^*\rho(x)U_0\Vert+6\eps<10\eps,\tag{$\triangle$,\ref{p1}}\]
and we conclude
\[\Vert 1_\mathcal M-U_1^*U_1\Vert\leq 2\mathbb E_x\Vert 1_\mathcal M-\ph(x)V_0^*\rho(x)^*U_1\Vert<20\eps.\tag{\ref{q1},\ref{r1}}\]
We proceed by estimating:
\begin{align*}
\Vert P-&U_0U_0^*\Vert=\Vert P-PUU^*P\Vert
\\&=\Vert(1_\infty-UU^*)^{1/2}P(1_\infty-UU^*)^{1/2}\Vert\tag{\ref{p3}}
\\&<\Vert (1_\infty-UU^*)^{1/2}A(1_\infty-UU^*)^{1/2}\Vert+2\eps\tag{$\triangle$,\ref{p1}}
\\&\leq\mathbb E_x\Vert(1_\infty-UU^*)^{1/2}\pi(x)^*UU^*\pi(x)(1_\infty-UU^*)^{1/2}\Vert+2\eps\tag{\ref{r1}}
\\&=\mathbb E_x\Vert U^*\pi(x)(1_\infty-UU^*)\pi(x)^*U\Vert +2\eps\tag{\ref{p3}}
\\&=\mathbb E_x\Vert U^*U-U^*\pi(x)UU^*\pi(x)^*U\Vert+2\eps
\\&\leq\mathbb E_x\Vert 1_\mathcal M-\tilde\ph(x)\tilde\ph(x)^*\Vert+2\eps<3\eps,\tag{\ref{p4}}
\end{align*}
so 
\begin{align*}
\Vert P-U_1U_1^*\Vert&\leq \Vert P-U_0U_0^*\Vert+\Vert (U_0-U_1)U_0^*\Vert+\Vert U_1(U_0-U_1)^*\Vert\tag{$\triangle$}
\\&<\Vert P-U_0U_0^*\Vert+6\eps+6\eps<15\eps.\tag{\ref{p1},\ref{p2}}
\end{align*}

Similarly, we replace $V_0=|V_0^*|T$ with $V_1:=\chi_{[1/2,1]}(|V_0^*|)T$ and get a partial isometry. In order to get the remaining estimates, we first note that
\begin{align*}
\Vert P-V_0V_0^*\Vert &=\mathbb E_x \Vert P-\rho(x)V_0V_0^*\rho(x)^*\Vert\tag{\ref{ui}}
\\&< \mathbb E_x\Vert U_0U_0^*-U_0U_0^*\rho(x)V_0V_0^*\rho(x)^*U_0U_0^*\Vert +9\eps\tag{$\triangle$}
\\&\leq \mathbb E_x\Vert 1_\mathcal M-U_0^*\rho(x)V_0V_0^*\rho(x)^*U_0\Vert+9\eps\tag{\ref{p1}}
\\&\leq 2\mathbb E_x\Vert 1_\mathcal M-\ph(x)V_0^*\rho(x)^*U_0\Vert+9\eps\tag{\ref{q1}}
\\&<17\eps,
\end{align*}
which entails that
\[\Vert V_1-V_0\Vert\leq 2\Vert |V_0^*|-|V_0^*|^2\Vert\leq 2\Vert P-V_0V_0^*\Vert<34\eps,\tag{\ref{p1},\ref{p4}}\]
so
\[\Vert P-V_1V_1^*\Vert\leq \Vert P-V_0V_0^*\Vert+68\eps<85\eps.\tag{$\triangle$,\ref{p1}}\]
Finally, we conclude that
\begin{align*}
\mathbb E_x\Vert 1_\mathcal M-&\ph(x)V_1^*\rho(x)^*U_1\Vert
\\&<\mathbb E_x\Vert 1_\mathcal M-\ph(x)V_0^*\rho(x)^*U_0\Vert+6\eps+34\eps\tag{$\triangle$,\ref{p1}}
<44\eps.
\end{align*}
Now the proof is complete by renaming $U_1$ and $V_1$ to $U$ and $V$.
\end{proof}

\begin{rem}
We remark that it follows from the estimates that
\begin{align*}
\mathbb E_x\Vert \rho(x)-U\ph(x)V^*\Vert&=\mathbb E_x\Vert P-U\ph(x)V^*\rho(x)^*\Vert\tag{\ref{ui}}\\
&\le \mathbb E_x\Vert UU^*-U\ph(x)V^*\rho(x)^*UU^*\Vert+30\eps\tag{$\triangle$}\\
&\leq \mathbb E_x\Vert 1_\mathcal M-\ph(x)V^*\rho(x)^*U\Vert+30\eps
<74\eps,\tag{\ref{p1}}
\end{align*}
that is, after inflating $\ph$ a bit, it is approximated on average by $\rho$.
\end{rem}
\begin{rem}\label{nemid}
We note for later use that if one does not require $V$ to be a partial isometry, then we can end the above proof earlier and get the better esitmate
\[\mathbb E_x\Vert 1_\mathcal M-\ph(x)V^*\rho(x)^*U\Vert<10\eps.\]
\end{rem}

\section{The inverse theorem}
In this section we explain how Theorem \ref{ghmain} follows from our Theorem \ref{main}. More precisely, in the case where $c=1-\eps$ is sufficiently close to $1$, the result is a direct consequence of Theorem \ref{main} in the case where $\Vert x\Vert=\Tr(x^*x)^{1/2}$ (we will, though, achieve a coarser lower bound on the trace), but for smaller $c$, although the techniques will be similar, we will need to accomodate our proof a bit.
Note that given a map $\ph\colon G\to\mathcal M$, expressions of the form
$$\mathbb E_{x}\mathbb E_{y}\mathbb E_z\tau(\ph(x)\ph(y)^*\ph(yz)\ph(xz)^*)$$ or
$$\mathbb E_{x}\mathbb E_{y}\mathbb E_z\tau(\ph(xy)\ph(y)^*\ph(z)\ph(xz)^*)$$ define non-negative real numbers. For the first, note that we have
$$\tau(\ph(x)\ph(y)^*\ph(yz)\ph(xz)^*) = \tau(\ph(y)^*\ph(yz)\ph(xz)^*\ph(x))$$
by the trace property. Now, it is easy to see from the proof of Proposition \ref{snakeispos} that $K \colon G \times G \to \mathcal M$ given by $$K(x,y) :=\mathbb E_z( \ph(y)^*\ph(yz)\ph(xz)^*\ph(x))$$ is a positive definite operator-valued kernel. Thus, by complete positivity of $\tau \colon M \to \mathbb C$, the composition $\tau \circ K \colon G \times G \to \mathbb C$ is a positive definite kernel as well. Using the representation theorem for positive definite kernels, there exists a Hilbert space valued function $\alpha \colon G \to \mathcal H$ such that for all $x,y \in G$, we have 
$(\tau \circ K)(x,y) = \langle \alpha(x),\alpha(y) \rangle$ and we can conclude that
\begin{align*}
&\mathbb E_{x}\mathbb E_{y}\mathbb E_z\tau(\ph(x)\ph(y)^*\ph(yz)\ph(xz)^*) \\
& =\mathbb E_{x}\mathbb E_{y}\mathbb E_z\tau(\ph(y)^*\ph(yz)\ph(xz)^*\ph(x))) \\
&= \mathbb E_{x}\mathbb E_{y} ((\tau \circ K)(x,y)) = \mathbb E_{x}\mathbb E_{y} \langle \alpha(x),\alpha(y) \rangle = \langle \mathbb E_{x}\alpha(x),\mathbb E_{y}\alpha(y) \rangle \geq 0.
\end{align*}
For the second expression, we compute more easily
\begin{align*}
&\mathbb E_{x}\mathbb E_{y}\mathbb E_z\tau(\ph(xy)\ph(y)^*\ph(z)\ph(xz)^*) \\
& =\mathbb E_{x}\tau(\mathbb E_{y}(\ph(xy)\ph(y)^*)\mathbb E_z(\ph(z)\ph(xz)^*)) \geq 0.
\end{align*}

We will now study what can be said about $\varphi$ in the presence of lower bounds on those quantities. Theorem \ref{invtrace1} will cover what is called the 99\% regime and Theorem \ref{invtrace2} covers what is called the 1\% regime.

\begin{theorem} \label{invtrace1}
Let $\eps>0$, let $G$ be an amenable group, let $\mathcal M$ be a von Neumann algebra and fix a normal trace $\tau$ on $\MM  $ with $\tau(1_\mathcal M)=1$. Let $\ph\colon G\to\mathcal M$ be a map with $\Vert\ph(x)\Vert_{\mathrm{op}}\leq 1$ for all $x\in G$. Assume that
\begin{gather*}
\mathbb E_{x}\mathbb E_{y}\mathbb E_z\tau(\ph(x)\ph(y)^*\ph(yz)\ph(xz)^*)\geq 1-\eps, \\
\mathbb E_{x}\mathbb E_{y}\mathbb E_z\tau(\ph(xy)\ph(y)^*\ph(z)\ph(xz)^*)\geq 1-\eps.\end{gather*}
Then there exist a projection $P\in\MM  $, partial isometries $U,V\in P\MM  1_\mathcal M$ and a representation $\rho\colon G\to \mathcal U(P\MM P)$ such that 
\begin{gather*}
\mathbb E_x\tau(\ph(x)V^*\rho(x)^*U)\geq 1-63\eps^{1/2},\\
\Vert 1_\mathcal M-U^*U\Vert_2<29\eps^{1/2},\quad \Vert P-UU^*\Vert_2<22\eps^{1/2},\quad \Vert P-VV^*\Vert_2<121\eps^{1/2}.
\end{gather*}
\end{theorem}
\begin{proof}
We consider the semi-norm $\Vert T\Vert_2:=\tau(T^*T)^{1/2},T\in\MM$. This is a ultraweakly lower semi-continuous unitarily invariant semi-norm on $\MM  $. 
It follows that
\begin{align*}
&\mathbb E_{x}\mathbb E_{y}\mathbb E_z\Vert 1_\mathcal M-\ph(x)\ph(y)^*\ph(yz)\ph(xz)^*\Vert_2
\\&\leq (\mathbb E_{x}\mathbb E_{y}\mathbb E_z(\Vert 1_\mathcal M-\ph(x)\ph(y)^*\ph(yz)\ph(xz)^*\Vert_2^2))^{1/2}
\\&\leq (\mathbb E_{x}\mathbb E_{y}\mathbb E_z(2-2\Ra\tau(\ph(x)\ph(y)^*\ph(yz)\ph(xz)^*))^{1/2}\leq (2\eps)^{1/2}.
\end{align*}
Similarly
\begin{align*}
\mathbb E_{x}\mathbb E_{y}\mathbb E_z\Vert 1_\mathcal M-\ph(xy)\ph(y)^*\ph(z)\ph(xz)^*\Vert_2\leq (2\eps)^{1/2}.
\end{align*}
Hence, from Theorem \ref{main} there exist $P\in\MM  $ and partial isometries $U,V\in P\MM  1_\mathcal M$ together with $\rho\colon G\to\mathcal U(P\MM P)$ such that
\[\Vert 1_\mathcal M-\EE_x\ph(x)V^*\rho(x)^*U\Vert_2\leq\EE_x\Vert 1_\mathcal M-\ph(x)V^*\rho(x)^*U\Vert_2<44(2\eps)^{1/2},\tag{\ref{r1}}\]
and furthermore $\Vert 1_\mathcal M-U^*U\Vert_2<20(2\eps)^{1/2}$, $\Vert P-UU^*\Vert_2<15(2\eps)^{1/2}$ and $\Vert P-VV^*\Vert_2<95(2\eps)^{1/2}$.
It now follows from the Cauchy-Schwarz ineqality that
\begin{align*}
\Big|1-|\tau(\EE_x\ph(x)V^*\rho(x)^*U)|\Big|&\leq |1-\tau(\EE_x\ph(x)V^*\rho(x)^*U)\big|
\\&=|\tau(1_\mathcal M\cdot(1_\mathcal M-\EE_x\ph(x)V^*\rho(x)^*U))|
\\&\leq \Vert 1_\mathcal M\Vert_2\Vert 1_\mathcal M-\EE_x\ph(x)V^*\rho(x)^*U\Vert_2
\\&<44(2\eps)^{1/2}.
\end{align*}
By multiplying $U$ with a complex number of modulus $1$, we can assume that $\tau(\EE_x\ph(x)V^*\rho(x)^*U)\geq 0$ so we get the desired
\[\tau(\EE_x\ph(x)V^*\rho(x)^*U)> 1-44(2\eps)^{1/2}>1-63\eps^{1/2}.\qedhere\]
\end{proof}

Now we turn our attention to the inverse theorem for general $c\in [0,1]$. This theorem is specific for the trace and we cannot use Theorem \ref{main}, but the proof is similar and even a bit shorter. Note that in this case we need less assumptions; we only assume one inequality.
\begin{theorem} \label{invtrace2}
Let $\mathcal M$ be a von Neumann algebra and let $\tau$ be a normal trace on $\MM  $ such that $\tau(1_\mathcal M)=1$.
Let $\ph\colon G\to\mathcal M$ be a map with $\Vert\ph(x)\Vert_{\mathrm{op}}\leq 1$ for all $x\in G$. Assume that
\[\mathbb E_{x}\mathbb E_{y}\mathbb E_z\tau(\ph(x)\ph(y)^*\ph(yz)\ph(xz)^*)\geq c.\]
Then there exists a projection $P\in\MM  $, partial isometries $U,V\in P(\MM  )1_\mathcal M$ and a representation $\rho\colon G\to \mathcal U(P\MM P)$ such that 
\begin{gather*}
\frac{c}{2}\leq \tau(UU^*)\leq \tau(P)\leq \frac{2}{c},\qquad
\frac{c}{2}\leq \tau(VV^*)\leq \tau(P)\leq \frac{2}{c},\\
\mathbb E_x\tau(\ph(x)V^*\rho(x)^*U)\geq \frac{c}{2}.
\end{gather*}
\end{theorem}
\begin{proof}
The proof begins as the proof of the main theorem. Define the positive definite $\tilde\ph(x):=\mathbb E_{y}\ph(xy)\ph(y)^*$ and use Proposition \ref{sdt} to determine $U\in \MM 1_\mathcal M$ and a representation $\pi\colon G\to \MM $ such that $\tilde\ph(x)=U^*\pi(x)U$.
Let $V:=\mathbb E_x\pi(x)^*U\ph(x)\in \MM  1_\mathcal M$. Also, let $A:=\mathbb E_x\pi(x)VV^*\pi(x)^*$,  which is a positive element of $\MM $, and let $P:=\chi_{[c/2,1]}(A^{1/2})$. Since $A$ clearly commutes with $\pi$, so does $P$, and therefore $\rho\colon G\to \mathcal U(P\MM P)$ given by $\rho(g)=P\pi(g)P$ for $g\in G$ is a representation.
We have that $\Vert A^{1/2}\Vert_{\mathrm{op}}^2=\Vert A\Vert_{\mathrm{op}}\leq \Vert V\Vert_{\mathrm{op}}^2\leq \Vert U\Vert_{\mathrm{op}}^2=\Vert \tilde\ph(1)\Vert\leq 1$ and by Kadison's inequality $\tau(A^{1/2})^2\leq\tau(A)=\tau(VV^*)=\tau(V^*V)\leq \tau(1_\mathcal M)=1$. Thus, by the inequality $\chi_{[c/2,1]}(t)\leq \frac{2}{c}t$ for $t\in [0,1]$, we get
\begin{align*}
\tau(P)\leq\frac{2}{c}\tau(A^{1/2})\leq\frac{2}{c},
\end{align*}
and by the inequality $t^2(1-\chi_{[c/2,1]}(t))\leq \frac{c}{2}t,$ for $t\in [0,1]$, we get
\begin{align*}
\tau(AP^\perp)\leq \frac{c}{2}\tau(A^{1/2})\leq \frac{c}{2},
\end{align*}
so, remembering that $\tau$ is normal and thus commutes with the mean, we conclude
\begin{align*}
\mathbb E_x\tau(\ph(x)V^*P^\perp\pi(x)^* U)
&=\mathbb E_x\tau(\pi(x)^*U\ph(x)V^*P^\perp)
\\&=\tau(VV^*P^\perp)=\mathbb E_y\tau(\pi(y)VV^*P^\perp\pi(y)^*)
\\&=\mathbb E_y\tau(\pi(y)VV^*\pi(y)^*P^\perp)=\tau(AP^\perp)\leq \frac{c}{2}.
\end{align*}

Furthermore, we have that
\begin{align*}
\mathbb E_x\tau(\ph(x)V^*\pi(x)^*U)&=\mathbb E_x\mathbb E_y\tau(\ph(x)\ph(y)^*U^*\pi(yx^{-1})U)
\\&=\mathbb E_x\mathbb E_y\tau(\ph(x)\ph(y)^*\tilde\ph(yx^{-1}))
\\&=\mathbb E_x\mathbb E_y\mathbb E_z\tau(\ph(x)\ph(y)^*\ph(yx^{-1}z)\ph(z)^*)
\\&=\mathbb E_x\mathbb E_y\mathbb E_z\tau(\ph(x)\ph(y)^*\ph(yz)\ph(xz)^*)\geq c,
\end{align*}

so we conclude
\begin{align*}
\mathbb E_x\tau(\ph(x)V^*\rho(x)^*U)&=\mathbb E_x\tau(\ph(x)V^*\pi(x)^*PU)\\
&=\mathbb E_x\tau(\ph(x)V^*\pi(x)^*U)-\mathbb E_x\tau(\ph(x)V^*\pi(x)^*P^\perp U)
\\&\geq c-\frac{c}{2}=\frac{c}{2}.
\end{align*}
As in the other proof, note that $U$ and $V$ are not partial isometries and they also fail to map into the right Hilbert space, so we have to correct for that.
The latter problem is again solved by replacing $U$ and $V$ with $U_0:=PU$ and $V_0:=PV$ which both lie in $P(\MM  )1_\mathcal M$. Of course, we still have
\[\mathbb E_x\tau(\ph(x)V_0^*\rho(x)^*U_0)=\mathbb E_x\tau(\ph(x)V^*\rho(x)^*U)\geq \frac{c}{2}.\]
Now since $\Vert U_0\Vert_{\mathrm{op}},\Vert V_0\Vert_{\mathrm{op}}\leq 1$, by Lemma \ref{partiso} we have that $U_0=\frac{1}{2}(U_1+U_2)$ and $V_0=\frac{1}{2}(V_1+V_2)$ where $U_1,U_2,V_1,V_2\in P\MM  1_\mathcal M$ are all partial isometries.
Thus, there must be at least one combination of partial isometries, say, $U_1$ and $V_1$ such that
\[\big|\mathbb E_x\tau(\ph(x)V_1^*\rho(x)^*U_1)\big|\geq \frac{c}{2}.\]
By multiplying $U_1$ with a complex number of modulus 1, we can assume
\[\mathbb E_x\tau(\ph(x)V_1^*\rho(x)^*U_1)\geq \frac{c}{2}.\]
Let $B:=\mathbb E_x\ph(x)V_1^*\rho(x)^*$. Then $\Vert B\Vert_{\mathrm{op}}\leq 1$, so $BB^*\leq 1_\mathcal M$, and hence $\tau(BB^*)\leq 1,$ which gives us
\[\tau(U_1U_1^*)\geq\tau(BB^*) \tau(U_1^*U_1)\geq |\tau(BU_1)|=\mathbb E_x\tau(\ph(x)V_1^*\rho(x)^*U_1)\geq\frac{c}{2}.\]
A similar calculation  gives us that $\tau(V_1V_1^*)\geq \frac{c}{2},$ and the proof is complete with $U_1$ and $V_1$ as $U$ and $V$.
\end{proof}

\section{Stability of $\eps$-representations}
We use our main theorem from to prove a stability result for $\eps$-representations. The proof actually works for a slightly larger class of maps.
\begin{defi}
If $G$ is amenable, a map $\ph\colon G\to H$ is called a \emph{mean} $\eps$\emph{-homomorphism} if for all $g\in G$
\[\mathbb E_hd(\ph(gh),\ph(g)\ph(h))<\eps.\]
\end{defi}
Note that the notion of a mean $\varepsilon$-homomorphism also covers the case of maps from a finite group to a discrete group which satisfy for all $g \in G$ that the equality $\ph(gh)=\ph(g)\ph(h)$ holds for most $h$.
Again we use the terminology \textit{a mean $\eps$-representation} if the group $H$ consists of operators.

\begin{theorem}\label{epso}
Let $\eps>0$, let $G$ be an amenable group, let $\mathcal M$ be a von Neumann algebra and let $\Vert\cdot\Vert$ be a unitarily invariant, ultraweakly lower semi-continuous semi-norm $\Vert\cdot \Vert$  on $\MM  $.
Let
$\ph\colon G\to \mathcal U(\mathcal M)$ be a mean $\eps$-representation with respect to the metric coming from $\Vert\cdot \Vert$. Then there is a projection $P\in\MM  $, a partial isometry $U\in P\MM  1_\mathcal M$ and a representation $\rho\colon G\to\mathcal U(P\MM P)$ such that
\[\Vert \ph(g)-U^*\rho(g)U\Vert<71\eps,\qquad g\in G,\]
and 
\[\Vert 1_\mathcal M-U^*U\Vert<40\eps,\qquad\Vert P-UU^*\Vert<30\eps.\]
\end{theorem}
\begin{proof}
We show that $\ph$ satisfies the conditions of our Theorem \ref{main} with $2\eps$ instead of $\eps$. 
\begin{align*}
\mathbb E_x&\mathbb E_y\mathbb E_z\Vert 1_\mathcal M-\ph(x)\ph(y)^*\ph(yz)\ph(xz)^*\Vert
\\&\leq \mathbb E_x\mathbb E_y\EE_z\Big(\Vert 1_\mathcal M-\ph(xy^{-1})\ph(yz)\ph(xz)^*\Vert
\\& \qquad+\Vert(\ph(xy^{-1})-\ph(x)\ph(y)^*)\ph(xy)\ph(xz)^*\Vert\Big)\tag{$\triangle$}
\\&<\EE_x\EE_y\EE_z\Vert 1_\mathcal M-\ph(xy^{-1})\ph(yz)\ph(xz)^*\Vert+\eps\tag{\ref{ui},\ref{li}}
\\&\leq\EE_x\EE_y\EE_z\Vert\ph(xz)-\ph(xy^{-1})\ph(yz)\Vert\tag{\ref{ui}}
+\eps
\\&=\EE_x\EE_y\EE_z\Vert \ph(xz)-\ph(x)\ph(z)\Vert+\eps<2\eps\tag{\ref{li}}
\end{align*}
and
\begin{align*}
\EE_x&\EE_y\EE_z\Vert 1_\mathcal M-\ph(xy)\ph(y)^*\ph(z)\ph(xz)^*\Vert
\\&\leq \EE_x\EE_y\EE_z\Big(\Vert 1_\mathcal M-\ph(x)\ph(z)\ph(xz)^*\Vert
\\&\qquad+\Vert (\ph(x)-\ph(xy)\ph(y)^*)\ph(z)\ph(xz)^*\Vert\Big)\tag{$\triangle$, \ref{ui}}
< 2\eps.
\end{align*}
It follows from Remark \ref{nemid} that 
there are a projection $P\in\MM  $ and operators $U,V\in P\MM  1_\mathcal M$, such that $U$ is a partial isometry and $\Vert V\Vert_{\mathrm{op}}\leq 1$, and a representation $\rho\colon G\to\mathcal M$ so that
\[\Vert 1_\mathcal M-\EE_x\ph(x)V^*\rho(x)^*U\Vert\leq\EE_x\Vert 1_\mathcal M-\ph(x)V^*\rho(x)^*U\Vert<20\eps\tag{\ref{r1}}\]
and
\[\Vert 1_\mathcal M-U^*U\Vert<40\eps,\qquad \Vert P-UU^*\Vert<30\eps.\]
Thus 
\begin{align*}
&\Vert \ph(g)-U^*\rho(g)U\Vert<\Vert \ph(g)\EE_x\ph(x) V^*\rho(x)^*U-U^*\rho(g)U\Vert+20\eps\tag{$\triangle$,\ref{ui}}
\\&\leq \EE_x\Vert\ph(g)\ph(x)V^*\rho(x)^*U-U^*\rho(g)U\Vert+20\eps\tag{\ref{r1}}
\\&< \EE_x\Vert \ph(gx)V^*\rho(x)^*U-U^*\rho(g)U\Vert+21\eps\tag{$\triangle$,\ref{p1}}
\\&= \EE_x\Vert\ph(x)V^*\rho(g^{-1}x)^*U-U^*\rho(g)U\Vert+21\eps\tag{\ref{li}}
\\&\leq \EE_x\Vert \ph(x)V^*\rho(x)^*UU^*\rho(g)U-U^*\rho(g)U\Vert+51\eps\tag{$\triangle$, \ref{p1}}
\\&<20\eps+51\eps=71\eps.\qedhere
\end{align*}
\end{proof}
\begin{rem}
In fact, the best way to prove this theorem might be to start from scratch and use the methods from the proof of Theorem \ref{main} accomodated suitably. This will, however, save us only $2\eps$ and we end up with the estimate $\Vert \ph(g)-U^*\rho(g)U\Vert<69\eps.$
\end{rem}

This general stability result for $\eps$-representations subsumes the theorems advertised in the introduction. These results follow almost immediately (but our estimates are different than the original ones). For convenience, we include some comments to the proofs here.

\begin{proof}[Proof of Theorem \ref{kazh}] We use Theorem \ref{epso} in the case where $\mathcal M\del\mathbb B(\mathcal H)$ is the von Neumann algebra generated by $\ph(G)$ and $\Vert\cdot\Vert$ is the operator norm $\Vert\cdot\Vert_{\mathrm{op}}$. If $\eps<\frac{1}{40}$, then $\Vert 1_\mathcal M-U^*U\Vert_{\mathrm{op}},\Vert P-UU^*\Vert_{\mathrm{op}}<1$ and since $1_\mathcal M-U^*U$ and $P-UU^*$ are projections, this implies $1_\mathcal M=U^*U$ and $P=UU^*$. It follows that $\rho'\colon G\to\mathcal U(\mathcal M)$ given by $\rho'(x)=U^*\rho(x)U$ is a unitary representation, and thus the result follows.
\end{proof}

\begin{proof}[Proof of Theorem \ref{gohat}] 
In the case $\mathcal M=\mathbb M_n$, we can identify $\MM$ with $\mathbb B(\ell^2(\NN))$ in such a way that $1_\mathcal M$ is a rank $n$ projection and we use Theorem \ref{epso} with $\Vert\cdot\Vert$ being the $2$-norm $\Vert\cdot\Vert_2$ coming from the semi-finite trace $\tau$ on $\mathbb B(\ell^2(\NN))$ normalized in such a way that rank $1$-projections have trace $\frac{1}{n}$. The inequalities $\Vert 1_\mathcal M-U^*U\Vert_2<40\eps,\Vert P-UU^*\Vert_2<30\eps$ translate into 
\[|\rank(P)-\rank(1_\mathcal M)|<(40^2+30^2)\eps^2n=2500\eps^2n.\]
First assume $\rank(P)\geq \rank(1_\mathcal M)$. Let $Q=P-UU^*$ and $R=1_\mathcal M-U^*U$. Since $U\in P(\MM)1_\mathcal M$ is a partial isometry, there is a partial isometry $U_0=Q(\MM)R$ such that $U':=U+U_0\in P(\MM)1_\mathcal M$ is an isometry.
\begin{align*}
&\Vert \ph(g)-(U')^*\rho(g)U'\Vert_2
\\&\leq \Vert \ph(g)-U^*\rho(g)U\Vert_2+\Vert U_0^*\rho(g)U\Vert_2+\Vert U^*\rho(g)U_0\Vert_2 +\Vert U_0^*\rho(g)U_0\Vert_2
\\&\leq 71\eps+3\Vert Q\Vert_2<161\eps.
\end{align*}

If $\rank(P)\leq \rank(1_\mathcal M)$, then we pick any projection $Q\geq P$ with $\rank(Q)=\rank(1_\mathcal M)$ and consider the representation 
\[\rho'\colon G\to Q(\MM)Q,\qquad \rho'(g)=\rho(g)+Q-P.\]
Since $U\in Q(\MM)1_\mathcal M$ is a partial isometry, it extends to a unitary $U'\colon 1_\mathcal M(\ell^2(\NN))\to Q(\ell^2(\NN))$ and we get that
\begin{align*}
\Vert \ph(g)-(U')^*\rho'(g)U'\Vert_2&\leq \Vert \ph(g)-U^*\rho(g)U^*\Vert_2+\Vert (U')^*(Q-P)U'\Vert_2
\\&<(71+\sqrt{2500})\eps=131\eps.
\end{align*}
In both cases we get the desired result.
\end{proof}

\section*{Acknowledgments}

Part of this research was carried out while the first author visited RIMS in Kyoto. He thanks this institution for its hospitality. The results of this article are part of the PhD thesis of the first author. This research was supported by ERC Starting Grant No.\ 277728 and ERC Consolidator Grant No.\ 681207.

\end{document}